\documentclass[11pt]{article}

 \usepackage{amsmath,amssymb,theorem,cases, comment}

\usepackage{color}

\numberwithin{equation}{section}

\newcommand {\Z}{\mathbb{Z}}

\newcommand {\R}{\mathbb{R}}
\newcommand {\C}{\mathbb{C}}

\newcommand{\be}{\begin{equation}}
\newcommand{\ee}{\end{equation}}

\newcommand{\eop}{\hspace*{\fill} $\Box$}

\newcommand{\qdim}[1]{\mathrm{qdim}[#1]}
\newcommand{\singlet}{\mathcal{W}(2,2p-1)}

\newtheorem{theorem}{Theorem}
\newtheorem{lemma}[theorem]{Lemma}
\newtheorem{proposition}[theorem]{Proposition}
\newtheorem{corollary}[theorem]{Corollary}
\newtheorem{conjecture}[theorem]{Conjecture}
\newtheorem{remark}[theorem]{Remark}
\newtheorem{definition}[theorem]{Definition}

\newenvironment{proof}[1][Proof]{\begin{trivlist}
\item[\hskip \labelsep {\bfseries #1}]}{\end{trivlist}}

\newcommand{\qed}{\nobreak \ifvmode \relax \else
      \ifdim\lastskip<1.5em \hskip-\lastskip
      \hskip1.5em plus0em minus0.5em \fi \nobreak
      \vrule height0.75em width0.5em depth0.25em\fi}

\begin{document}
\title{\bf False theta functions and the Verlinde formula}

\author{Thomas Creutzig and Antun Milas } \maketitle

\abstract{

We discover new analytic properties of classical partial and false theta functions and their potential applications to representation theory 
of $\mathcal{W}$-algebras and vertex algebras in general.  More precisely, motivated by clues from conformal field theory, first, we are able to determine modular-like transformation properties of {\em regularized} partial and false theta functions. Then, after suitable identification of regularized partial/false theta functions with the characters of atypical modules for the singlet vertex algebra $\mathcal{W}(2,2p-1)$, we formulate a Verlinde-type formula for the fusion rules of irreducible $\mathcal{W}(2,2p-1)$-modules. 
}

\section{Introduction: partial and false theta functions}

In this paper, we are primarily concerned with modular-like transformation properties of functions
\begin{equation} \label{ptf}
P_{a,b}(u,\tau)=\sum_{n=0}^\infty z^{n+\frac{b}{2a}} q^{a(n+\frac{b}{2a})^2},  \ \ q=e(\tau), \ \ z=e(u),
\end{equation}
where $a,b \in 
\mathbb{N}$, $\tau \in \mathbb{H}$, $u \in \mathbb{C}$,  called {\em partial theta functions} 
\cite{AB}\footnote{In \cite{AB} a slightly different definition was used.}. If $z$ is specialized to be $q^c$, we call them partial 
theta series.
The names explain themselves as the usual Jacobi theta functions/series are also given by (\ref{ptf}), but with the summation over 
all integers. We will also be interested in closely 
related series called {\em false theta series}, where the summation is over $\mathbb{Z}$, but the sign 
choice does not correspond to any specialization of the theta function. A typical example is 
\begin{equation} \label{fts}
\sum_{n \in \mathbb{Z}} {\rm sgn}(n) q^{a(n+\frac{b}{2a})^2},
\end{equation}
which plays  a prominent role in our work.  Actually, the only false theta series appearing in this paper 
are simply differences of two partial thetas.

While theta functions of course enjoy modular 
transformation properties, and are related to numerous concepts in mathematics and theoretical physics, partial/false theta functions do not seem  to have any modular properties and their relevance in mathematics is somewhat obscure and random.
Most prominently, false/partial thetas appear in various identities involving $q$-hypergeometric series and even some partition identities (for a thorough selection of results see \cite{AB} and references therein). 
They also seem to arise in computation of topological invariants. 
For example, in \cite{LZ}, the (rescaled) Witten-Reshetikhin-Turaev (WRT) invariant associated to the homology spheres was studied as radial limiting value of certain partial/false theta series. Also,  it was recently shown 
in \cite{GL} (see Section 14)
that the generating series of colored Jones polynomials for alternating knots are  given by (\ref{fts}).
But as far as we know, there were no serious attempts to relate false theta functions to concepts in infinite-dimensional representation theory and to ideas in conformal field theory (eg. Verlinde formula). We also point out that partial/false theta functions are {\em not}  mock theta functions as studied in \cite{Z}, although there is a connection (see  \cite{Za} for instance). 

In this paper we take a radical different point of view to (\ref{ptf}) and eventually (\ref{fts}).
Our starting observation is that some of the series discussed earlier are essentially (i.e. up to Dedekind $\eta$-function factor) graded dimensions of modules for the vertex operator algebra $\mathcal{W}(2,2p-1)$, also called the singlet algebra (this was also noticed by Flohr in  \cite{Fl}). This vertex algebra did not attract so much attention, primarily because it is not $C_2$-cofinite, although it is instrumental for studying more interesting triplet vertex algebra \cite{FHST}, \cite{FGST1}, \cite{FGST2}, \cite{NT}, \cite{TW}, \cite{AdM1}, \cite{CF}, \cite{Fl}, etc.  Characters of modules for the triplet are well-understood; they can be organized 
so that they form a vector-valued (logarithmic) modular form. Using this approach a Verlinde-type formula can be also obtained  \cite{FHST} (in the rational setup see \cite{Ve}, and the proof of the Verlinde conjecture by Huang \cite{Hu}).  
Thus, it is very natural to ask: (i) Do irreducible characters of the singlet algebra also  obey modular-like transformations properties, and (ii) is there a 
Verlinde-type formula for irreps based on these properties?
Of course, because we have infinitely many irreps, in modular transformation formulas we also allow integral part as in other works 
on "continuous"  Verlinde-type formulas \cite{AC},  \cite{BR}, \cite{CR1}, \cite{CR2}, \cite{CR3}.
In order to describe the family of singlet algebras $\mathcal{W}(2,2p-1)$ 
parameterized by integers $p\geq 2$ we introduce the numbers $\alpha_+=\sqrt{2p}, \alpha_-=-\sqrt{2/p}$ and $\alpha_0=\alpha_++\alpha_-$.
For our purposes we first note that the singlet algebra admits two types of $\mathbb{Z}_{\geq 0}$-graded  representations:
\begin{itemize}
\item[(1)] generic (or typical). Fock space representations $F_{\lambda}$, with the character
$$\text{ch}[{F}_{\lambda}](\tau)=\frac{q^{\frac{1}{2}(\lambda-\alpha_0/2)^2}}{\eta(\tau)}; \ \ \  \lambda \in \mathbb{C} $$
and $\eta(\tau)=q^{\frac{1}{24}} \prod_{i=1}^\infty (1-q^i)$ is the usual Dedekind eta-function.
\item[(2)] Non-generic (or atypical).  Certain subquotients or reducible Fock spaces, denoted by $M_{r,s}$, with the character
$$\text{ch}[M_{r,s}](\tau)=\frac{ P_{p,pr-s}(0,\tau)-P_{p,pr+s}(0,\tau)}{\eta(\tau)},$$ 
where $P_{a,b}$ is as in (\ref{ptf}). The range is $r \in \mathbb{Z}$ and $1 \leq s \leq p$ \footnote{Strictly speaking $M_{r,p}$ modules can be also viewed as typical representations. For a precise definition of typical/atypical see Definition 18.}.
\end{itemize}
Modular transformation properties of (1) are easily computed via Gauss' integral, but those of (2)
are much more delicate. As a remedy,  we introduce the $\epsilon$-regularized characters $\text{ch}[X^\epsilon] $  (see Section 2 for details; see also \cite{Fl})
such that $\lim_{\epsilon \to 0}\text{ch}[X^\epsilon]=\text{ch}[X]$. These regularized atypical characters transform much nicer
as illustrated in our first main results of the paper. 

\begin{theorem}
The modular $S$-transformation of atypical characters is
\begin{equation*}
\begin{split}
\text{ch}[M^\epsilon_{r,s}]\Bigl(-\frac{1}{\tau}\Bigr)&=
\int_{\mathbb R} S^\epsilon_{(r,s),\mu+\alpha_0/2}\mathrm{ch}[ F^\epsilon_{\mu+\alpha_0/2}](\tau)d\mu +X^\epsilon_{r,s}(\tau)
 \end{split}
\end{equation*}
with
\[ 
 S^\epsilon_{(r,s),\mu+\alpha_0/2}= -e^{-2\pi \epsilon((r-1)\alpha_+/2+\mu)}e^{\pi i (r-1)\alpha_+\mu}
\frac{\sin\bigl(\pi s \alpha_-(\mu+i\epsilon)\bigr)}{\sin\bigl(\pi \alpha_+(\mu+i\epsilon)\bigr)}\\
\]
and
\[
X^\epsilon_{r,s}(\tau)= \frac{1}{4 \eta(\tau)}({\rm sgn}(\mathrm{Re}(\epsilon))+1)\sum_{n\in\mathbb Z}(-1)^{rn}
e^{\pi i \frac{s}{p}n}q^{\frac{1}{2}(\frac{n^2}{\alpha_+^2}-\epsilon^2)}
\bigl(q^{-i\epsilon \frac{n}{\alpha_+}}-q^{i\epsilon \frac{n}{\alpha_+}}\bigr).
\]
\end{theorem}
Although this result concerns certain characters of modules it relies on another key result for the partial theta function (see Theorem \ref{thm:partialS}).
Interestingly, the previous theorem does not provide us with the usual modular transformation properties one would expect in non-rational theories due to the theta-like 
term $X^\epsilon_{r,s}(\tau)$ that has no obvious interpretation as a regularized character. Surprisingly, the term disappears  for  ${\rm Re}(\epsilon)<0$, which we assume to hold.  By using the above $S$-transformation formulas we define 
 a suitable product on the space of characters by mimicking the  Verlinde formula for the fusion rules. We obtain 
 the following result
\begin{theorem} (Verlinde-type formula) With parametrization of irreps as in Section 4, and 
with multiplication of regularized characters as in (\ref{rig-double}), we have 
\begin{equation*}
\begin{split}
\mathrm{ch}[F_\lambda^\epsilon]\times  \mathrm{ch}[F_\mu^\epsilon] &= \sum_{\ell=0}^{p-1}\mathrm{ch}[F_{\lambda+\mu+\ell\alpha_-}^\epsilon]\\
\mathrm{ch}[M_{r,s}^\epsilon]\times  \mathrm{ch}[F_\mu^\epsilon] &= 
\sum_{\substack{\ell=-s+2\\ \ell+s=0\, \mathrm{mod}\, 2}}^{s}\mathrm{ch}[F_{\mu+\alpha_{r,\ell}}^\epsilon]\\
\mathrm{ch}[M_{r,s}^\epsilon]\times  \mathrm{ch}[M_{r',s'}^\epsilon] &= 
\quad\sum_{\substack{\ell=|s-s'|+1\\ \ell+s+s'=1\, \mathrm{mod}\, 2}}^{min \{ s+s'-1,p \}}\mathrm{ch}[M_{r+r'-1,\ell}^\epsilon] \\
& +\sum_{\substack{\ell=p+1\\ \ell+s+s'=1\, \mathrm{mod}\, 2}}^{s+s'-1}\Bigl(\mathrm{ch}[M_{r+r'-2,\ell-p}^\epsilon]+
\mathrm{ch}[M_{r+r'-1,2p-\ell}^\epsilon]+\mathrm{ch}[M_{r+r',\ell-p}^\epsilon]\Bigr).
\end{split}
\end{equation*}
\end{theorem}
The algebra structure on the integer span of characters given by this Verlinde-type formula we also call
Verlinde algebra (of characters). Notice that the product does not depend on the regularization parameter $\epsilon$. Independence of the choice of this parameter for the final answer is exactly
the requirement for a good regularization in mathematical physics. 

At last, in parallel with the triplet algebra \cite{FGST1,FGST2}, and partially motivated by \cite{KL}, we expect the category of (ordinary) $\singlet$-modules 
to be equivalent to the category of finite-dimensional representations for a certain infinite-dimensional quantum group. This conjecture is amplified with the computation of regularized quantum dimensions in Section 4.3 (see Theorem \ref{Q-thm}). We hope to return to the
problem of identifying the relevant quantum group in our future publications.

David Ridout and Simon Wood have informed us that they are preparing
a manuscript on the Verlinde formula of the $(p,p')$-singlet algebra \cite{RW}. Instead of an analytic approach
they follow the strategy of the previous works on the Verlinde formula, see e.g. \cite{CR4} for an introduction and the example of $\mathcal{W}(2,3)$, of resolving atypical modules in terms of typical ones.  
They find agreement with our Verlinde formula when restricting to the case of $\mathcal{W}(2,2p-1)$. Very recently, in \cite{CMW}, we extended methods of this paper to 
obtain rigorous derivation of results in \cite{RW}.

{\bf Acknowledgements:} We thank K. Bringmann  on some discussions related to this paper. We also 
thank D. Ridout and S. Wood for very useful discussion, suggestions and (S.W.) for pointing out some 
inconsistencies in a previous version of the paper. 
T.C. appreciates that D.R. shared some computations on the singlet algebra. 
T.C. is supported by NSERC Discovery Grant RES0020460.

\section{Modularity of regularized partial theta functions}

In this part we study modular-like transformation properties of a partial theta function, which can be used 
to express more complicated partial and eventually false theta functions. Let 
$$P(u,\tau)=\sum_{k\in\mathbb Z_{\geq 0}+\frac{1}{2}}z^k q^{k^2/2}.$$
where $q=e(\tau)$, $\tau \in \mathbb{H}$, the upper half-plane, and $z=e(u)$, where $u \in \mathbb{C}$. 
This function is obviously holomorphic. If one tries to compute modular transformation properties under $(u, \tau) \mapsto (\frac{u}{\tau},-\frac{1}{\tau})$, some divergent integrals quickly appear. To fix this inconvenience we use a method sometimes used in the physics literature, and in particular in \cite{Fl},  called {\em regularization}. The idea is to deform the "charge"  variable $u$ by introducing an additional parameter denoted by $\epsilon$. On one hand $\epsilon$ can be viewed as the contour deformation parameter, but also as a quantum group parameter (the two seem to be connected). In this paper we do not try to make this connection precise, leaving it for future considerations.

\begin{definition}
{\em Let $\epsilon\in \mathbb C\setminus i\mathbb R$, then the partial regularized theta function is
$$ P_\epsilon(u,\tau)= \sum_{k\in\mathbb Z_{\geq 0}+\frac{1}{2}}z^ke^{2\pi \epsilon k}q^{k^2/2}.$$ }
\end{definition}
The main result of this section will be
\begin{theorem}\label{thm:partialS}
Let $\epsilon\in \mathbb C\setminus i\mathbb R$, then the modular properties of the regularized partial theta function are:
\begin{equation*}
 \begin{split}
&  P_\epsilon(u,\tau+1)=e^{\pi i/4}P_\epsilon(u,\tau), \\
& P_\epsilon((u/\tau,-1/\tau)=\frac{e^{\pi i u^2/\tau}\sqrt{-i\tau}}{2}\Bigl(
-i\int_{\R} \frac{q^{x^2/2}z^{x}}{\sin (\pi(  x+i\epsilon))}\, dx +\frac{1}{2}(\mathrm{sgn}(\mathrm{Re}(\epsilon))+1)\vartheta_{4,\epsilon}(u,\tau)\Bigr),
 \end{split}
\end{equation*}
with the regularized ordinary theta function
$$\vartheta_{4,\epsilon}(u,\tau)= \sum_{n\in\mathbb Z}(-1)^nz^{n-i\epsilon}q^{(n-i\epsilon)^2/2}.$$
\end{theorem}
The modular T-transformation ($ \tau \mapsto \tau +1$) is obvious, and the rest of this section will be devoted to proving the modular S-transformation of the theorem.

Let us start with its elliptic transformation properties.
\begin{proposition}\label{prop:ellipepsilon}
The elliptic transformations of $ P_\epsilon$ are
\begin{equation*}
\begin{split}
&  P_\epsilon(u+1,\tau)=- P_\epsilon(u,\tau), \\
&  P_\epsilon(u+\tau,\tau)=z^{-1}q^{-1/2}e^{-2\pi\epsilon} P_\epsilon(u,\tau)-z^{-1/2}q^{-3/8}e^{-\pi\epsilon}\,.
\end{split}
\end{equation*}
\end{proposition}
\begin{proof}
A straightforward rewriting of the sums. \eop
\end{proof}
As a corollary, we get
\begin{corollary}\label{cor:ellipepsilon}
Let $\tilde u=u/\tau$ and $\tilde\tau=-1/\tau$, then
\begin{equation*}
\begin{split}
& P_\epsilon((u+1)/\tau,-1/\tau)=\tilde z\tilde q^{-1/2}e^{2\pi\epsilon}P_\epsilon(\tilde u,\tilde \tau)+\tilde z^{1/2}\tilde q^{-3/8}e^{\pi\epsilon}
\\ 
& P_\epsilon((u+\tau)/\tau,-1/\tau)=-P_\epsilon(\tilde u,\tilde \tau)\,.
\end{split}
\end{equation*}
\end{corollary}
\begin{proof} This follows from Proposition \ref{prop:ellipepsilon} as follows.
Let $\tilde u=u/\tau$, $\tilde\tau=-1/\tau$ and $u'=\tilde u-\tilde\tau$ and correspondingly 
$\tilde z=e^{2\pi i\tilde u},\tilde q=e^{2\pi i\tilde\tau}, z'=e^{2\pi iu'}$, then 
$$ P_\epsilon((u+1)/\tau,-1/\tau)=P_\epsilon(\tilde u-\tilde\tau,\tilde\tau)=P_\epsilon(u',\tilde\tau)$$
and hence with Proposition \ref{prop:ellipepsilon}
$$P_\epsilon(u',\tilde\tau)=z'\tilde q^{1/2}e^{2\pi\epsilon}P_\epsilon(u'+\tilde\tau,\tilde\tau)+z'^{1/2}\tilde q^{1/8}e^{\pi\epsilon}
=\tilde z\tilde q^{-1/2}e^{2\pi\epsilon}P_\epsilon(\tilde u,\tilde\tau)+\tilde z^{1/2}\tilde q^{-3/8}e^{\pi\epsilon}.$$
The second equation is obvious. \eop \end{proof}
\begin{proposition}\label{prop:feps}
Let $\gamma(u,\tau)=e^{-\pi i u^2/\tau}$, then $f_\epsilon(u,\tau)=\gamma(u,\tau)P_\epsilon(\tilde u,\tilde \tau)$ satisfies
\begin{align}
f_\epsilon(u,\tau)-e^{-2\pi\epsilon}f_\epsilon(u+1,\tau)\ &=\ -e^{-\pi\epsilon}\gamma(u+1/2,\tau)\\
f_\epsilon(u,\tau)+zq^{1/2}f_\epsilon(u+\tau,\tau)\ &= \ 0\,.
\end{align}
\end{proposition}
\begin{proof}
This follows from corollary \ref{cor:ellipepsilon}. \eop
\end{proof}
\begin{lemma}\label{lemma:gauss}
Let $\tilde z=e(u/\tau)$, $\tilde q=e(-1/\tau)$, then we have
$$ \tilde z^k\tilde q^{\frac{k^2}{2}} \ = \ \sqrt{-i\tau}e^{\frac{\pi i u^2}{\tau}}\int_{-\infty}^\infty q^{\frac{w^2}{2}}z^we^{-2\pi i wk}\,dw\,.$$
\end{lemma}
\begin{proof}
This is a Gauss integral.\eop
\end{proof}
\begin{proposition}\label{prop:heps}
Let
$$ h_\epsilon(u,\tau)\ = \ -\frac{i\sqrt{-i\tau}}{2}\int_{\mathbb R} \frac{q^{x^2/2}z^{x}}{\sin (\pi(  x+i\epsilon))}\, dx,$$
then $h_\epsilon$ satisfies 
\begin{align}
h_\epsilon(u,\tau)-e^{-2\pi\epsilon}h_\epsilon(u+1,\tau)\ &=\ -e^{-\pi\epsilon}\gamma(u+1/2,\tau)\\
h_\epsilon(u,\tau)+zq^{1/2}h_\epsilon(u+\tau,\tau)\ &= \ 0\,.
\end{align}
Also note, that since $\epsilon \notin i\mathbb R$, the integral is absolutely convergent for all $u\in\mathbb C, \tau\in \mathbb H$.
\end{proposition}
\begin{proof}
The left-hand side of the first equation is a Gauss integral, so the equation follows with Lemma \ref{lemma:gauss}.
The second equation follows with the substitution $x\rightarrow x+1$ in the second integral.
\eop\end{proof}
We define the correction 
$$ R_\epsilon=f_\epsilon-h_\epsilon.$$
\begin{proposition}
The correction term is holomorphic and satisfies the following functional equation
\begin{align}
R_\epsilon(u,\tau)-e^{-2\pi\epsilon}R_\epsilon(u+1,\tau)\ &=\ 0\\
R_\epsilon(u,\tau)+zq^{1/2}R_\epsilon(u+\tau,\tau)\ &= \ 0\,.
\end{align}
\end{proposition}
\begin{proof} Follows from Proposition \ref{prop:feps} and \ref{prop:heps}. \eop \end{proof}
\begin{proposition}
The Jacobi-theta-like function (for fixed $\epsilon\in\mathbb C$)
$$\vartheta_{4,\epsilon}(u,\tau)= \sum_{n\in\mathbb Z}(-1)^nz^{n-i\epsilon}q^{(n-i\epsilon)^2/2}$$
satisfies 
\begin{align}
\vartheta_{4,\epsilon}(u,\tau)-e^{-2\pi\epsilon}\vartheta_{4,\epsilon}(u+1,\tau)\ &=\ 0\\
\vartheta_{4,\epsilon}(u,\tau)+zq^{1/2}\vartheta_{4,\epsilon}(u+\tau,\tau)\ &= \ 0\,
\end{align}
and up to a scalar multiple it is the unique holomorphic function in $u$ with this property.
\end{proposition}
\begin{proof}
Well-known fact for standard theta-functions, but $q^{-\epsilon^2/2}z^{-i\epsilon}\vartheta_{4,0}(u-i\epsilon\tau,\tau)=\vartheta_{4,\epsilon}(u,\tau)$
can be expressed in terms of the standard Jacobi-theta function $\vartheta_{4,0}(u,\tau)$.\eop
\end{proof}
We define $\alpha_\epsilon(\tau)$ by $R_\epsilon(u,\tau)=\alpha_\epsilon(\tau) \vartheta_{4,\epsilon}(u,\tau)$.
\begin{proposition}\label{prop:taudep}
$\alpha_\epsilon(\tau)=\alpha_\epsilon\sqrt{-i\tau}$
where $\alpha_\epsilon$ is independent of $\tau$.
\end{proposition}
\begin{proof}
Define $$ \Delta = \frac{1}{\pi i }\partial_\tau -\frac{1}{(2\pi i)^2}\partial^2_u,$$
then we compute that 
$$\Delta \frac{f_\epsilon(u,\tau)}{\sqrt{-i\tau}}\ = \ \Delta \frac{h_\epsilon(u,\tau)}{\sqrt{-i\tau}}\ = \ \Delta\vartheta_{4,\epsilon}(u,\tau)\ = \ 0$$ 
now, let $\sqrt{-i\tau}\beta_\epsilon(\tau)=\alpha_\epsilon(\tau)$, then the statement implies
$\partial_\tau\beta_\epsilon(\tau)=0$.\eop
\end{proof}
\begin{proposition}\label{prop:epsdep}
There exists $\alpha\in\mathbb C$, such that
\begin{equation*}
\begin{split}
\alpha_\epsilon \ = \ \begin{cases} \alpha &\mathrm{if} \ \mathrm{Re}(\epsilon)>0\\
\alpha-1 &\mathrm{if}\ \mathrm{Re}(\epsilon)<0\end{cases}.
\end{split}
\end{equation*}
\end{proposition}
\begin{proof}
Let $\mu\in\C$ with $\text{Re}(\epsilon)\neq -\text{Im}(\mu)$, then rewriting the appropriate sums yields
$$ z^\mu q^{\mu^2/2}\vartheta_{4,\epsilon-i\mu}(u+\mu\tau,\tau)=\vartheta_{4,\epsilon}(u,\tau)\qquad,\qquad 
z^\mu q^{\mu^2/2}f_{\epsilon-i\mu}(u+\mu\tau,\tau)=f_{\epsilon}(u,\tau).$$
With the substitution $y=x+\mu$, we get
$$z^\mu q^{\mu^2/2}h_{\epsilon-i\mu}(u+\mu\tau,\tau)= -\frac{i\sqrt{-i\tau}}{2}\int_{\mathbb R} \frac{q^{x^2/2+\mu x+\mu^2/2}z^{x+\mu}}{\sin (\pi(  x+\mu+i\epsilon))}dx= -\frac{i\sqrt{-i\tau}}{2}\int_{\mathbb R+\mu} \frac{q^{y^2/2}z^{y}}{\sin (\pi(  y+i\epsilon))}dy.$$
Since $\tau$ is in the upper half plane, we can thus write the difference $h_{\epsilon}(u,\tau)-z^\mu q^{\mu^2/2}h_{\epsilon-i\mu}(u+\mu\tau,\tau)$ as a contour integral over the contour $\mathcal C_\mu$ which connects the two areas of integration of the previous equation,
\begin{equation*}
\begin{split}
h_{\epsilon}(u,\tau)-z^\mu q^{\mu^2/2}h_{\epsilon-i\mu}(u+\mu\tau,\tau)\, &=\,
-\frac{i\sqrt{-i\tau}}{2}\int_{\mathcal C_\mu} \frac{q^{x^2/2}z^{x}}{\sin (\pi(  x+i\epsilon))}\, dx.
\end{split}
\end{equation*}
Let $\epsilon'=\epsilon-i\mu$, we can then evaluate the integral by the residuum theorem and get
\begin{equation*}
\begin{split}
-\frac{i\sqrt{-i\tau}}{2}\int_{\mathcal C_\mu} \frac{q^{x^2/2}z^{x}}{\sin (\pi(  x+i\epsilon))}\, dx\, &=\, \sqrt{-i\tau}\sum_{n\in\mathbb Z}(-1)^nz^{n-i\epsilon}q^{(n-i\epsilon)^2/2} \begin{cases}
0 &\mathrm{if}\ \mathrm{Re}(\epsilon)>0\ , \ \mathrm{Re}(\epsilon')>0\\
1 &\mathrm{if}\ \mathrm{Re}(\epsilon)>0\ , \ \mathrm{Re}(\epsilon')<0\\
-1 &\mathrm{if}\ \mathrm{Re}(\epsilon)<0\ , \ \mathrm{Re}(\epsilon')>0\\
0 &\mathrm{if}\ \mathrm{Re}(\epsilon)<0\ , \ \mathrm{Re}(\epsilon')<0\\
\end{cases}\\
&=\, \sqrt{-i\tau}\vartheta_{4,\epsilon}(u,\tau) \begin{cases}
0 &\mathrm{if}\ \mathrm{Re}(\epsilon)>0\ , \ \mathrm{Re}(\epsilon')>0\\
1 &\mathrm{if}\ \mathrm{Re}(\epsilon)>0\ , \ \mathrm{Re}(\epsilon')<0\\
-1 &\mathrm{if}\ \mathrm{Re}(\epsilon)<0\ , \ \mathrm{Re}(\epsilon')>0\\
0 &\mathrm{if}\ \mathrm{Re}(\epsilon)<0\ , \ \mathrm{Re}(\epsilon')<0\\
\end{cases}
\end{split}
\end{equation*}
but this implies 
\begin{equation*}
\begin{split}
\alpha_\epsilon \ = \ \begin{cases} \alpha &\mathrm{if} \ \mathrm{Re}(\epsilon)>0\\
\alpha-1 &\mathrm{if}\ \mathrm{Re}(\epsilon)<0\end{cases}
\end{split}
\end{equation*}
for some $\alpha\in\C$.\eop
\end{proof}
It remains to compute $\alpha$. 
\begin{definition}
{\em We define the regularized false theta function
$$F\vartheta_\epsilon(u,\tau)=P_\epsilon(u,\tau)-P_{-\epsilon}(-u,\tau).$$ }
\end{definition}
\begin{proposition}\label{prop:falseS}
The modular S-transformation of the regularized false theta function is
$$ F\vartheta_\epsilon(u/\tau,-1/\tau)=e^{\pi i u^2/\tau}\sqrt{-i\tau}\Bigl(
-i\int_{\R} \frac{q^{x^2/2}z^{x}}{\sin (\pi(  x+i\epsilon))}\, dx +\mathrm{sgn}(\mathrm{Re}(\epsilon))\vartheta_{4,\epsilon}(u,\tau)\Bigr)
$$
\end{proposition}
\begin{proof}
This follows directly from $h_{-\epsilon}(-u,\tau)=-h_\epsilon(u,\tau)$, $\vartheta_{4,-\epsilon}(-u,\tau)=\vartheta_{4,\epsilon}(u,\tau)$, $\alpha_\epsilon-\alpha_{-\epsilon}=\mathrm{sgn}(\mathrm{Re}(\epsilon))$
and the expression for $\alpha_\epsilon(\tau)$ that is Proposition \ref{prop:taudep} and \ref{prop:epsdep}.\eop
\end{proof}
\begin{proposition}
$\alpha=1$
\end{proposition}
\begin{proof}
Define $\vartheta_{2,\epsilon}(u,\tau)= \sum_{k\in\mathbb Z+\frac{1}{2}}z^ke^{2\pi \epsilon k}q^{k^2/2}$.
Then a short computation gives 
$$ \vartheta_{2,\epsilon}(u/\tau,-1/\tau)=e^{\pi iu^2/\tau}\sqrt{-i\tau}\vartheta_{4,\epsilon}(u,\tau).$$
The partital regularized theta function satisfies $2P_{\epsilon}(u,\tau)=F\vartheta_\epsilon(u,\tau)+\vartheta_{2,\epsilon}(u,\tau)$ and hence with Proposition \ref{prop:falseS} and above equation the claim follows.\eop
\end{proof}
This proposition completes the proof of Theorem \ref{thm:partialS}.

\begin{remark} 
{\em Clearly,
\[ P_{a,b}(u,\tau)= \sum_{n=0}^\infty z^{n+\frac{b}{2a}} q^{a(n+\frac{b}{2a})^2}=
z^{\frac{b}{2a}-\frac{1}{2}}q^{a(\frac{b}{2a}-\frac{1}{2})^2}P(u+(b-a)\tau,2a\tau)\]
so the previous theorem solves the problem of finding the modular-like transformations of the non-generic
family of regularized singlet algebra characters.}
\end{remark}

\section{The singlet vertex algebra $\mathcal{W}(2,2p-1)$}

Let  $$\hat{\mathfrak{h}}=\mathfrak{h} \otimes \mathbb{C}[t,t^{-1}]+\mathbb{C} k.$$ 
denote the rank one Heisenberg Lie algebra. We choose its generators 
to be $\varphi(n)$, $n \in \mathbb{Z}$, such that $[\varphi(n),\varphi(m)]=\delta_{m+n,0} k$. 
Denote by ${F}_{\lambda} \cong U(\hat{\mathfrak{h}}_-)$, $\lambda \in \mathbb{C}$,  the usual Fock space of charge 
$\lambda$, so that $\varphi(0) \cdot e^{\lambda \varphi} =\lambda e^{\lambda \varphi}$, 
where $e^{\lambda \varphi}$ is a lowest weight vector for ${F}_{\lambda}$. We fix the level to be one, i.e. 
the central element $k$ will act by multiplication
with one on all Fock spaces. For every $p \geq 2$, we choose the conformal vector to be
$$\omega=\frac{1}{2} \varphi(-1)^2{\bf 1}+\frac{p-1}{\sqrt{2p}}\varphi(-2){\bf 1} \in {F}_0.$$
This equips $F_0$ with a VOA structure of central charge $c_{p,1}=1-6\frac{(p-1)^2}{p}$.
We shall be using standard parametrization of $(1,p)$ lowest weight logarithmic minimal models following 
$$h_{m,n}=\frac{(np-m)^2-(p-1)^2}{4p}.$$
In this parametrization we may assume $1 \leq m \leq p$ and $n \in \mathbb{Z}$.


Next, we introduce the singlet vertex algebra after Kausch \cite{Ka}, where we prefer to follow the approach from \cite{A}, \cite{AdM1} and \cite{AdM3} (see also \cite{FHST}, \cite{FGST1}, \cite{FGST2}). 

\subsection{Definition of $\mathcal{W}(2,2p-1)$}
Again, here $p \in \mathbb{N}_{\geq 2}$. Denote by 
$$V_L=\bigoplus_{\lambda \in \sqrt{2 p}  \mathbb{Z}} {F}_{\lambda} $$ the lattice 
vertex algebra associated to rank one even lattice $\sqrt{2 p}  \mathbb{Z}$ \cite{AdM3} (see also \cite{LL}).
Consider its dual lattice 
$\widetilde{L} = {\mathbb{Z}}(\frac{1}{\sqrt{2 p}})$. Then we have a generalized vertex algebra structure
on  
$$V_{ \widetilde{L}}=\bigoplus_{\lambda \in \frac{1 }{\sqrt{2p}} \mathbb{Z}} {F}_{\lambda}=\bigoplus_{i=0}^{2p-1} V_{L+\frac{i}{\sqrt{2p}}}.$$
The vertex algebra $V_L$ is then a
vertex subalgebra of $V_{ \widetilde{L}}$.

The element $L(0)$ of the Virasoro algebra defines a $
{\mathbb{N}}$--gradation on $V_L$. As
in \cite{Ka} define the following {\em long} and {\em short} screening operators

$$Q= e^{ \sqrt{2p} \varphi} _0, \qquad \widetilde{Q} =e^{-\sqrt{\frac{2}{p}} \varphi}_0,$$ 
respectively, where
we use 
$$e^{\gamma}(x)=\sum_{n \in \mathbb{Z}} e^{\gamma}_n x^{-n-1},$$
the Fourier expansion of $e^{\gamma}$. Then we  have 
$$ [Q,\widetilde{Q}] =0, \ \ [L(n), Q] = [L(n), \widetilde{Q}] = 0
\ \ (n \in {\Z}).$$
Thus, the operators $Q$ and $\widetilde{Q}$ are intertwinners 
among Virasoro algebra
modules. In fact, the Virasoro vertex operator algebra
$L(c_{p,1},0) \subset \mathcal{F}_0$ is the kernel of the screening
operator $Q$. Define
$$\mathcal{W}(2,2p-1) = \mbox{Ker}_{{F}_0 } \widetilde{Q},$$
called the {\em singlet vertex algebra}.
\noindent Since $ \widetilde{Q}$ commutes with the action of the Virasoro
algebra, we have
$$L(c_{p,1},0) \subset  \mathcal{W}(2,2p-1). $$
%


The vertex operator algebra  $\mathcal{W}(2,2p-1) $ is completely
reducible as a Virasoro algebra module and the following decomposition
holds:
\begin{equation}
\singlet= \bigoplus_{n =0} ^{\infty} U(Vir).\ u^{(n)}  =
\bigoplus_{n =0} ^{\infty} L(c_{p,1}, n^{2} p + n p - n),\nonumber
\end{equation}
where \begin{equation} \label{un}
 u ^{(n)} = Q^{n} e^{-n \sqrt{2p} \varphi}.
\end{equation}
In addition,  $ \mathcal{W}(2,2p-1)$ is strongly generated by
$\omega$ and the primary vector \be \label{H} H= Q e^{-\sqrt{2p} \varphi} \ee of conformal
weight $2p-1$.

\subsection{Irreducible $\mathcal{W}(2,2p-1)$-modules}

Complete classification of all (weak) irreducible $\mathcal{W}(2,2p-1)$-modules is presently unknown.  
On the other hand, $\mathbb{Z}_{\geq 0}$-gradable irreducible modules were classified in \cite{A} (see also \cite{AdM1} for some additional details). From now on we consider finitely generated $\mathbb{Z}_{\geq 0}$-gradable ordinary modules  whose characters are well-defined.
%
We do not consider logarithmic modules in this work.
First we distinguish between typical and atypical modules.
\begin{definition} {\em An irreducible ($\mathbb{Z}_{\geq 0}$-graded)  $\mathcal{W}(2,2p-1)$-module is called {\em typical} if it remains irreducible as a Virasoro module, and {\em atypical} otherwise.}
\end{definition}

We denote by $\text{ch}[X](\tau)$ the usual character of $X$, the trace of $q^{L(0)-c/24}$.
Generic characters are easily computed. Denote by $F_{\lambda}$ the Fock space of charge $\lambda$ as in the previous section.
Then clearly, 
$${\rm ch}[F_{\lambda}](\tau)=\frac{q^{h_{\lambda}-c_{p,1}/24}}{(q;q)_\infty}=\frac{q^{(\lambda-\alpha_0/2)^2/2}}{\eta(\tau)}.$$
We also observe the symmetry
$${\rm ch}[F_{\lambda}](\tau)={\rm ch}[F_{\alpha_0-\lambda}](\tau).$$

By using results from \cite{A} and  \cite{AdM1}, we easily infer that all atypical irreducible $\mathcal{W}(2,2p-1)$-modules can be constructed as subquotients of ${F}_{\lambda}$, where $\lambda \in \sqrt{2p} \mathbb{Z}+\frac{i}{\sqrt{2p}}$, $0 \leq i \leq 2p-1$ (the dual lattice $\tilde{L}$). Every such Fock space yields a unique irreducible $\mathcal{W}(2,2p-1)$-module.
To see this, we shall first slightly adjust the parametrization of $\tilde{L}$.
Let $\alpha_+=\sqrt{2p}$, $\alpha_-=-\sqrt{2/p}$ and let also $$\alpha_0=\alpha_++\alpha_-.$$
Further let 
$$ \alpha_{r,s}=-\frac{1}{2}(r\alpha_++s\alpha_--\alpha_0)=-\frac{r-1}{2} \sqrt{2p}+\frac{s-1}{\sqrt{2p}} \in \tilde{L}.$$
The range for $r$ is the set of integers and $1 \leq s \leq p$.
Now to each $F_{\alpha_{r,s}}$ we associate an irreducible module $M_{r,s}$. 

It is known that for $s=p$, the singlet module $F_{\alpha_{r,p}}$ is irreducible \cite{AdM1}. So for $r \in \mathbb{Z}$, we let 
$$M_{r,p}:=F_{\alpha_{r,p}}.$$
From now on we may assume the range  to be $1 \leq s \leq p-1$. The module $F_{\alpha_{r,s}}$ has a composition series of length $2$ with respect to $\mathcal{W}(2,2p-1)$. 
We denote by 
$$M_{r,s}:={\rm soc}(F_{\alpha_{r,s}}).$$
The socle can be also taken with respect to the Virasoro algebra though.  We first consider the case $r \geq 1$. Then $M_{r,s}$ is of the same highest weight as 
$F_{\alpha_{r,s}}$. We clearly have a short exact sequence
$$0 \rightarrow M_{r,s} \rightarrow F_{\alpha_{r,s}} \rightarrow N_{r,s}  \rightarrow 0,$$
where $N_{r,s}$ is another irreducible module. By using well-known formulas for decomposition of $M_{r,s}$ into irreducible Virasoro modules \cite{CRW}, \cite{AdM1}  we easily obtain
\be \label{char+1}
{\rm ch}[{M_{r,s}}](\tau)=\frac{1}{\eta(\tau)} \left( \sum_{n=0}^\infty q^{p(\frac{r}{2}+n-\frac{s}{2p})^2}-q^{p(\frac{r}{2}+n+\frac{s}{2p})^2} \right).
\ee
Now we focus on $F_{\alpha_{r,s}}$, $r \leq 0$. Similarly, we get 

 $${\rm ch}[M_{r,s}]
=\frac{1}{\eta(\tau)} \left(\sum_{n=0}^\infty q^{p(-\frac{r}{2}+\frac{1}{2}+n+\frac{p-s}{2p})^2}-q^{p(-\frac{r}{2}+\frac{1}{2}+n+\frac{p+s}{2p})^2}\right).$$
Observe now the relation (for $ r \leq 0$)
$$ \sum_{n=0}^\infty q^{p(-\frac{r}{2}+\frac{1}{2}+n+\frac{p-s}{2p})^2}-q^{p(-\frac{r}{2}+\frac{1}{2}+n+\frac{p+s}{2p})^2}= \sum_{n=0}^\infty q^{p(\frac{r}{2}+n-\frac{s}{2p})^2}-q^{p(\frac{r}{2}+n+\frac{s}{2p})^2},$$
due to cancellations in the second sum.
To summarize, for  $r \in \mathbb{Z}$ and $1 \leq s \leq p$ we have 
$$\text{ch}[M_{r,s}](\tau)=\frac{ P_{p,p r -s}(0,\tau)-P_{p,p r +s}(0,\tau)}{\eta(\tau)},$$ 
where $P_{a,b}(\tau)$ is as in the introduction. In particular, for $M_{1,1}=\mathcal{W}(2,2p-1)$,  we get
$$\text{ch}[\mathcal{W}(2,2p-1)](\tau)=\frac{\sum_{n \in \mathbb{Z}} {\rm sgn}(n) q^{p(n+\frac{p-1}{2p})^2}}{\eta(\tau)}.$$

\begin{remark} \label{felder}
In addition to considerations coming from decomposition of Fock spaces into Virasoro algebra modules it is also useful to apply Felder's resolution in the category of $\mathcal{W}(2,2p-1)$-modules
(see formula (2.26) \cite{CRW}, for instance):
$$ 
 \cdots \rightarrow F_{\alpha_{r,s}} \xrightarrow{\tilde{Q}^{[s]}} F_{\alpha_{r+1,p-s}} \xrightarrow{\tilde{Q}^{[p-s]}} F_{\alpha_{r+2,s}} \xrightarrow{\tilde{Q}^{[s]}} F_{\alpha_{r+3,p-s}} \rightarrow \cdots,
$$
where $\tilde{Q}^{[s]}$ are suitable "powers" of the short screening operator $\tilde{Q}$.
It can be shown $M_{r,s}={\rm Ker} \ \tilde{Q}^{[s]} \subset F_{\alpha_{r,s}}$ \cite{CRW}, so by the Euler-Poincar\'e principle we easily get 
$$\text{ch}[M_{r,s}](\tau)=\sum_{n= 0}^\infty \text{ch}[ F_{\alpha_{r-2n-1,p-s}}](\tau)
-\text{ch}[ F_{\alpha_{r-2n-2,s}}](\tau).$$
This formula will be useful in the next section
\end{remark}

\subsection{Regularized characters of $\mathcal{W}(2,2p-1)$-modules}


Now, we define the regularized characters by introducing a parameter $\epsilon$.  We let
\begin{equation}\label{eq:char}
\begin{split}
\text{ch}[ F^\epsilon_\lambda](\tau)&=e^{2\pi\epsilon(\lambda-\alpha_0/2)}\frac{q^{(\lambda-\alpha_0/2)^2/2}}{\eta(\tau)} \\
\text{ch}[M^\epsilon_{r,s}](\tau)&=\sum_{n= 0}^\infty \text{ch}[ F^\epsilon_{\alpha_{r-2n-1,p-s}}](\tau)
-\text{ch}[ F^\epsilon_{\alpha_{r-2n-2,s}}](\tau)
 \end{split}
\end{equation}




Observe that typical $\epsilon$-regularized characters are simply ${\rm tr}_{F_\lambda} e^{2 \pi \epsilon (\varphi(0)-\alpha_0/2)} q^{L(0)-c/24}$. But atypical 
regularization is more subtle although very natural in view of Remark \ref{felder}.

Note, that the characters of the atypical modules are parameterized by $r,s \in\Z$ with $1\leq s\leq p$.
Also, $\text{ch}[M^\epsilon_{r,0}](\tau)=0$ and in the case $s>p$, $\text{ch}[M^\epsilon_{r,s}](\tau)$ is actually an integral combination of atypical module characters
(virtual character):
\begin{proposition}\label{prop:rel}
The regularized typical and atypical characters satisfy the following relations
\begin{equation*}
\begin{split}
 \mathrm{ch}[ F^\epsilon_{\alpha_{r,s}}](\tau)&= \mathrm{ch}[M^\epsilon_{r,s}](\tau)+\mathrm{ch}[M^\epsilon_{r+1,p-s}](\tau)
 \end{split}
 \end{equation*}
 while $\mathrm{ch}[M^\epsilon_{r,s}](\tau)$ for $p < s \leq 2p-1$ \footnote{Similar formulas can be obtained for higher $s$ but we do not need them right now.} is the following linear combination of atypical module characters
\begin{equation*}
  \begin{split}
 \mathrm{ch}[M^\epsilon_{r,s}](\tau)&= \mathrm{ch}[M^\epsilon_{r-1,s-p}](\tau)+\mathrm{ch}[M^\epsilon_{r,2p-s}](\tau)+\mathrm{ch}[M^\epsilon_{r+1,s-p}](\tau).
 \end{split}
 \end{equation*}
\end{proposition}
\begin{proof}
 The first equality follows directly from \eqref{eq:char}, while for the second one, we use
 that $\alpha_{r,s}=\alpha_{r+1,s+p}$ for all $r,s\in\Z$ to
 compute
 \begin{equation*}
  \begin{split}
   \text{ch}[M^\epsilon_{r,s}](\tau)&=\sum_{n= 0}^\infty \text{ch}[ F^\epsilon_{\alpha_{r-2n-1,p-s}}](\tau)
-\text{ch}[ F^\epsilon_{\alpha_{r-2n-2,s}}](\tau)\\
&= \sum_{n= 0}^\infty \text{ch}[ F^\epsilon_{\alpha_{r-2n,2p-s}}](\tau)
-\text{ch}[ F^\epsilon_{\alpha_{r-2n-3,s-p}}](\tau)\\
&= \text{ch}[ F^\epsilon_{\alpha_{r,2p-s}}](\tau)+\text{ch}[M^\epsilon_{r-1,s-p}](\tau)\\
&= \text{ch}[M^\epsilon_{r-1,s-p}](\tau)+\text{ch}[M^\epsilon_{r,2p-s}](\tau)+\text{ch}[M^\epsilon_{r+1,s-p}](\tau). 
\qquad\qquad\qquad\qquad\qquad\qquad\text{\eop}
\end{split}
 \end{equation*}
\end{proof}

\begin{proposition}\label{prop:charfalse}
Let $\beta^\pm_{r,s}=((r-1)\alpha_+\pm s\alpha_-)/2$, then the atypical characters are
\begin{equation*}
 \mathrm{ch}[M^\epsilon_{r,s}](\tau)= \mathrm{ch}[ F^\epsilon_{\alpha_0/2-\beta^-_{r,s}}](\tau)P_{\alpha_+\epsilon}(-\alpha_+\beta^-_{r,s}\tau;\alpha_+^2\tau)-
 \mathrm{ch}[ F^\epsilon_{\alpha_0/2-\beta^+_{r,s}}](\tau)P_{\alpha_+\epsilon}(-\alpha_+\beta^+_{r,s}\tau;\alpha_+^2\tau)
\end{equation*}
\end{proposition}
\begin{proof}
 This is a straightforward rewriting. \eop
\end{proof}

\subsection{Modular properties of characters}

\begin{proposition}
The modular S-transformation of typical characters is
\begin{equation*}
\begin{split}
\mathrm{ch}[ F^\epsilon_{\lambda+\alpha_0/2}]\bigl(\frac{-1}{\tau}\bigr)&=
\int_{\mathbb R} S_{\lambda+\alpha_0/2,\mu+\alpha_0/2}^\epsilon \mathrm{ch}[F^\epsilon_{\mu+\alpha_0/2}](\tau)d\mu,
\end{split}
\end{equation*}
with $S_{\lambda+\alpha_0/2,\mu+\alpha_0/2}^\epsilon= e^{2\pi\epsilon(\lambda-\mu)}e^{-2\pi i\lambda\mu}$.
 \end{proposition}
\begin{proof}
 This follows from the Gauss integral of Lemma \ref{lemma:gauss}. \eop
\end{proof}
\begin{proposition}
The modular S-transformation of atypical characters is
\begin{equation*}
\begin{split}
\text{ch}[M^\epsilon_{r,s}]\Bigl(-\frac{1}{\tau}\Bigr)&=
\int_{\mathbb R} S^\epsilon_{(r,s),\mu+\alpha_0/2}\mathrm{ch}[F^\epsilon_{\mu+\alpha_0/2}](\tau)d\mu +X^\epsilon_{r,s}(\tau)
 \end{split}
\end{equation*}
with
\[ 
 S^\epsilon_{(r,s),\mu+\alpha_0/2}= -e^{-2\pi \epsilon((r-1)\alpha_+/2+\mu)}e^{\pi i (r-1)\alpha_+\mu}
\frac{\sin\bigl(\pi s \alpha_-(\mu+i\epsilon)\bigr)}{\sin\bigl(\pi \alpha_+(\mu+i\epsilon)\bigr)}\\
\]
and
\[
X^\epsilon_{r,s}(\tau)= \frac{1}{4 \eta(\tau)}({\rm sgn}(\mathrm{Re}(\epsilon))+1)\sum_{n\in\mathbb Z}(-1)^{rn}
e^{\pi i \frac{s}{p}n}q^{\frac{1}{2}(\frac{n^2}{\alpha_+^2}-\epsilon^2)}
\bigl(q^{-i\epsilon \frac{n}{\alpha_+}}-q^{i\epsilon \frac{n}{\alpha_+}}\bigr).
\]
\end{proposition}
Note, that in the limit $\epsilon\rightarrow 0$, $X^\epsilon_{r,s}$ vanishes. 
\begin{proof}
With Theorem \ref{thm:partialS}, we get 
\begin{equation*}
\begin{split}
\mathrm{ch}[F^\epsilon_{\alpha_0/2-\beta^\pm_{r,s}}]\Bigl(\frac{-1}{\tau}\Bigr)P_{\alpha_+\epsilon}\Bigl(\frac{\alpha_+\beta^\pm_{r,s}}{\tau};-
\frac{\alpha_+^2}{\tau}\Bigr)&=
\frac{1}{2i}\int_{\mathbb R}\frac{e^{-2\pi \epsilon\mu}e^{2\pi i \beta_{r,s}^\pm(\mu+i\epsilon)}\mathrm{ch}[ F^\epsilon_{\mu+\alpha_0/2}](\tau)}{\sin\bigl(\pi \alpha_+(\mu+i\epsilon)\bigr)}d\mu+\\
&\qquad \frac{1}{4 \eta(\tau)}({\rm sgn}(\mathrm{Re}(\epsilon))+1)\vartheta_{4,\alpha_+\epsilon}\Bigl(\frac{\beta^\pm_{r,s}}{\alpha_+};\frac{\tau}{\alpha_+^2}\Bigr)
\end{split}
\end{equation*}
and hence with Proposition \ref{prop:charfalse} the statement follows. \eop
\end{proof}

\section{A Verlinde-type formula}
Let us consider the case ${\rm Re}(\epsilon)<0$ so there is no correction term present. 
We are interested in applying the Verlinde formula. This requires a unitary S-matrix (actually, $S$-kernel), which is spoiled by the regularization and instead
we have
\begin{equation} \label{non-unitary}
 "\int_{\mathbb R}" S^\epsilon_{\lambda\mu}\overline{S^{-\bar\epsilon}_{\mu\nu}}d\mu=
"\int_{\mathbb R}"e^{-2\pi i \mu(\lambda-\nu)}e^{2\pi \epsilon(\lambda-\nu)}d\mu=e^{2\pi \epsilon(\lambda-\nu)}\delta(\lambda-\nu)=\delta(\lambda-\nu),
\end{equation}
where $\delta(x-y)$ is the Dirac delta-function supported at $x=y$.
We thus define the regularized fusion coefficients
\begin{equation} \label{fusion-coeff}
{N^\epsilon_{ab}}^c = "\int_{\mathbb R}" \frac{S^\epsilon_{a\rho}S^\epsilon_{b\rho}\overline{S^{-\bar\epsilon}_{\ c \rho}}}{S^\epsilon_{(1,1)\rho}}d\rho 
\end{equation}
where $(1,1)$ refers to the vacuum module $M_{1,1}$.
Consider the vector space $\mathcal{V}_{ch}$  generated by $\mathrm{ch}[V^\epsilon]$ where $V=M_{r,s}$ or $V=F_\lambda$. 

We want to turn $\mathcal{V}_{ch}$ 
into a commutative associative algebra (called the {\em Verlinde algebra of characters}) by defining the product to be 
\begin{equation} \label{Verlinde-char}
\mathrm{ch}[V^\epsilon_a]\times  \mathrm{ch}[V^\epsilon_b] := "\int_{\mathbb R}" {N^\epsilon_{ab}}^c\mathrm{ch}[V^\epsilon_c]dc
\end{equation}
and extending this multiplication by linearity. 
We expect the right hand side to be a finite sum with non-negative integer multiplicities (as in the Verlinde formula). 

\subsection{Making the Verlinde algebra of characters rigorous}

The integration in (\ref{non-unitary}) and (\ref{fusion-coeff}) over $\mathbb{R}$ should not be taken literally. As we shall see shortly, the function that we would like to 
integrate is clearly non-integrable. Yet, as we are primarily interested in (\ref{Verlinde-char}), and not so much (\ref{fusion-coeff}), we explain first how to make (\ref{Verlinde-char}) rigorous and then how to view  (\ref{fusion-coeff}) not as a numerical quantity but rather as  a distribution.
Thus instead of  working with
\begin{equation} \label{nonrig-double}
\int_{\mathbb R} \left( \int_{\mathbb{R}}   \frac{S^\epsilon_{a\rho}S^\epsilon_{b\rho}\overline{S^{-\bar\epsilon}_{\rho \mu}}}{S^\epsilon_{(1,1)\rho}}d\rho      \right)\mathrm{ch}[F^\epsilon_\mu] d\mu,
\end{equation}
we redefine the fusion product in the Verlinde algebra of characters as
\begin{equation} \label{rig-double}
\int_{\mathbb R} \left( \int_{\mathbb{R}}   \frac{S^\epsilon_{a\rho}S^\epsilon_{b\rho}\overline{S^{-\bar\epsilon}_{\rho \mu}}}{S^\epsilon_{(1,1)\rho}}      \mathrm{ch}[F^\epsilon_\mu] d\mu \right) d \rho 
\end{equation}
This double integral turns out to be well-defined in our examples.

To see this let us start from the classical Fourier inversion formula. Suppose that $f(x)$ and its Fourier transform $\hat{f}(x)$ lie in an appropriate $L^1$-space. Then we have
$$f(x)=\int_{\mathbb{R}} (\int_{\mathbb{R}} e^{2 \pi i (x-y)z} f(y) dy) dz.$$
Going back to (\ref{rig-double}), in this setup the test functions are essentially
$$f(\mu)=e^{2 \pi \epsilon (\mu+c)} q^{(\mu-\alpha_0/2)^2/2} e^{2 \pi \epsilon(\mu-\alpha_0/2)},$$
where we ignore the Dedekind $\eta$ denominator and $c$ does not depend on $\mu$. We are integrating
\begin{equation} \label{ff}
\int_{\mathbb{R}} ( \int_{\mathbb{R}} e^{-2 \pi i \rho(\mu + c)} f(\mu) d \mu) d \rho=f(-c)=q^{(-c-\alpha_0/2)^2/2} e^{2 \pi \epsilon(-c-\alpha_0/2)}.
\end{equation}
Notice that the same result can be inferred  by working with (\ref{nonrig-double}) and by using (heuristic) $\delta$-function
$$\delta(x-y)=\int_{\mathbb{R}}  e^{2 \pi i \rho (x-y) } d \rho,$$
as in (\ref{non-unitary}).
Then second integration, against the delta function (now viewed as  distribution) is simply evaluation so we obtain the same result as in 
(\ref{ff}). 
To handle infinite sums (see below) we only have to notice that for every $Re(\epsilon)<0$ and $f(\mu)$ as before, we have 
$$\int_{\mathbb{R}} \frac{f(\mu)}{{\rm sin}(\mu+\epsilon i)} d\mu=2 i \sum_{m=0}^\infty \int_{\mathbb{R}} f(\mu) e^{\epsilon(2m+1)- i \mu (2m+1)} d \mu.$$
We first fix $\epsilon=\epsilon_1+i \epsilon_2$, where $\epsilon_1<0$ and $\epsilon_2 \in \mathbb{R}$.
To prove the last formula we first observe that $|q^{(\lambda-\alpha_0/2)^2}|=|q|^{(\lambda-\alpha_0/2)^2}$ 
with $|q|<1$. Now, Gauss' integral formula shows that 
$\sum_{m=0}^\infty \int_{\mathbb{R}} |f(\mu) e^{(\epsilon_1+\epsilon_2 i)(2m+1)- i \mu (2m+1)}| d \mu= \sum_{m=0}^\infty \int_{\mathbb{R}} |f(\mu) e^{\epsilon_1 (2m+1)}| d \mu$ is convergent. Finally, by Fubini's 
theorem we can interchange the sum and integration. This section clarifies all  future computations involving the formal delta function.

\subsection{Verlinde  algebra of characters}
In the next theorem we explicitly determine this algebra.
\begin{theorem} \label{thm:verlinde}
Let $1\leq s,s'\leq p$, then the Verlinde algebra of characters is associative and commutative and is given by
\begin{equation*}
\begin{split}
\mathrm{ch}[F_\lambda^\epsilon]\times  \mathrm{ch}[F_\mu^\epsilon] &= \sum_{\ell=0}^{p-1}\mathrm{ch}[F_{\lambda+\mu+\ell\alpha_-}^\epsilon]\\
\mathrm{ch}[M_{r,s}^\epsilon]\times  \mathrm{ch}[F_\mu^\epsilon] &= 
\sum_{\substack{\ell=-s+2\\ \ell+s=0\, \mathrm{mod}\, 2}}^{s}\mathrm{ch}[F_{\mu+\alpha_{r,\ell}}^\epsilon]\\
\mathrm{ch}[M_{r,s}^\epsilon]\times  \mathrm{ch}[M_{r',s'}^\epsilon] &= 
\quad\sum_{\substack{\ell=|s-s'|+1\\ \ell+s+s'=1\, \mathrm{mod}\, 2}}^{min \{ s+s'-1,p \}}\mathrm{ch}[M_{r+r'-1,\ell}^\epsilon] \\
& +\sum_{\substack{\ell=p+1\\ \ell+s+s'=1\, \mathrm{mod}\, 2}}^{s+s'-1}\Bigl(\mathrm{ch}[M_{r+r'-2,\ell-p}^\epsilon]+
\mathrm{ch}[M_{r+r'-1,2p-\ell}^\epsilon]+\mathrm{ch}[M_{r+r',\ell-p}^\epsilon]\Bigr),
\end{split}
\end{equation*}
where $\displaystyle{\sum_{l=i}^j } (\cdot)=0$ for $i>j$.
\end{theorem}

\begin{remark}
{\em In mathematical physics a regularization is introduced to avoid divergent quantities.
It is then required that the final result is independent of the regularization scheme.
In our case the final result is the Verlinde algebra, which indeed is independent of
the choice of our regularization parameter $\epsilon$.}
\end{remark}

\begin{proof}
We first note the following identity
\begin{equation*}
 \frac{\text{sin}(sx)}{\text{sin}(x)}=\sum_{\substack{\ell=-s+1\\ \ell+s=1\, \mathrm{mod}\, 2 }}^{s-1}e^{ix\ell}
\end{equation*}
which is verified by multiplying both sides with $\text{sin}(x)$.
It follows the Verlinde fusion of typicals with themselves. Note that $\alpha_+=-p\alpha_-$.
\begin{equation*}
 \begin{split}
  {N^\epsilon_{\lambda+\alpha_0/2,\mu+\alpha_0/2}}^{\nu+\alpha_0/2} &= -\int_{\mathbb R} e^{-2\pi i (\rho+i\epsilon)(\lambda+\mu-\nu)} 
\frac{\text{sin}(\pi \alpha_+(\rho+i\epsilon))}{\text{sin}(\pi \alpha_-(\rho+i\epsilon))}d\rho\\
&= \int_{\mathbb R} \sum_{\substack{\ell=-p+1\\ \ell+p=1\, \mathrm{mod}\, 2}}^{p-1} e^{-2\pi i (\rho+i\epsilon)(\lambda+\mu-\nu+\ell\alpha_-/2)}d\rho \\
&= \sum_{\substack{\ell=-p+1\\ \ell+p=1\, \mathrm{mod}\, 2}}^{p-1}\delta(\lambda+\mu-\nu+\ell\alpha_-/2) e^{2\pi\epsilon(\lambda+\mu-\nu+\ell\alpha_-/2)} \\
&= \sum_{\substack{\ell=-p+1\\ \ell+p=1\, \mathrm{mod}\, 2}}^{p-1}\delta(\lambda+\mu-\nu+\ell\alpha_-/2)
 \end{split}
\end{equation*}
and hence  
\[
\mathrm{ch}[F_{\lambda+\alpha_0/2}^\epsilon]\times  \mathrm{ch}[F_{\mu+\alpha_0/2}^\epsilon] = 
\sum_{\substack{\ell=-p+1\\ \ell+p=1\, \mathrm{mod}\, 2}}^{p-1}\mathrm{ch}[F_{\lambda+\mu+\ell\alpha_-/2+\alpha_0/2}^\epsilon] 
\]
so that  
\begin{equation*}
\begin{split}
\mathrm{ch}[F_\lambda^\epsilon]\times  \mathrm{ch}[F_\mu^\epsilon] &= 
\sum_{\substack{\ell=-p+1\\ \ell+p=1\, \mathrm{mod}\, 2}}^{p-1}\mathrm{ch}[F_{\lambda+\mu+\ell\alpha_-/2-\alpha_0/2}^\epsilon]
= \sum_{\ell=0}^{p-1}\mathrm{ch}[F_{\lambda+\mu+\ell\alpha_-}^\epsilon]
\end{split}
 \end{equation*}
since $-\alpha_0=(p-1)\alpha_-$. 
The Verlinde fusion of atypicals with typicals is proven in the same manner as the previous case.
\begin{equation*}
 \begin{split}
  {N^\epsilon_{(r,s), \mu+\alpha_0/2}}^{\nu+\alpha_0/2} &= \int_{\mathbb R} e^{-2\pi i (\rho+i\epsilon)(\mu-(r-1)\alpha_+/2-\nu)} 
\frac{\text{sin}(\pi s\alpha_-(\rho+i\epsilon))}{\text{sin}(\pi \alpha_-(\rho+i\epsilon))}d\rho\\
&= \int_{\mathbb R} \sum_{\substack{\ell=-s+1\\ \ell+s=1\, \mathrm{mod}\, 2}}^{s-1} e^{-2\pi i (\rho+i\epsilon)(\mu-(r-1)\alpha_+/2-\nu+\ell\alpha_-/2)} d\rho\\
&= \sum_{\substack{\ell=-s+1\\ \ell+s=1\, \mathrm{mod}\, 2}}^{s-1}\delta(\mu-(r-1)\alpha_+/2-\nu+\ell\alpha_-/2)\\
&= \sum_{\substack{\ell=-s+1\\ \ell+s=1\, \mathrm{mod}\, 2}}^{s-1}\delta(\alpha_{r,\ell+1}+\mu-\nu)
 \end{split}
\end{equation*}
Recall that $\alpha_{r,s}=-\frac{1}{2}(r\alpha_++s\alpha_--\alpha_0)$.
It follows  that
\[
\mathrm{ch}[M_{r,s}^\epsilon]\times  \mathrm{ch}[F_{\mu+\alpha_0/2}^\epsilon] = 
\sum_{\substack{\ell=-s+2\\ \ell+s=0 \, \mathrm{mod}\, 2}}^{s}\mathrm{ch}[F_{\mu+\alpha_0/2+\alpha_{r,\ell}}^\epsilon]. 
\]
Finally, the case of atypicals uses the following identity for $Im(x)<0$, and positive integers $s,s',p$
\begin{equation*}
 -\frac{\text{sin}(sx)\text{sin}(s'x)}{\text{sin}(x)\text{sin}(px)}=\sum_{\ell'=0}^\infty e^{-i px(2\ell'+1)}
\sum_{\substack{\ell=|s'-s|+1\\ s+s'+\ell=1\,\text{mod}\, 2}}^{s'+s-1}
\Bigl(e^{-i x\ell} -e^{i x\ell} \Bigr),
\end{equation*}
 $Im(x)<0$ ensures convergence and the identity is verified by multiplying both sides with the denominator of the left-hand side. 
We thus get 
\begin{equation*}
 \begin{split}
  {N^\epsilon_{(r,s)(r',s')}}^{\nu+\alpha_0/2} &= -\int_{\mathbb R} e^{-\pi i (\rho+i\epsilon)(-(r+r'-2)\alpha_+-2\nu)} 
\frac{\text{sin}(\pi s\alpha_-(\rho+i\epsilon))\text{sin}(\pi s'\alpha_-(\rho+i\epsilon))}{\text{sin}(\pi \alpha_-(\rho+i\epsilon))\text{sin}(\pi \alpha_+(\rho+i\epsilon))}d\rho\\
&\hspace*{-5mm}= \int_{\mathbb R} \sum_{\ell'=0}^\infty 
\sum_{\substack{\ell=|s'-s|+1\\ s+s'+\ell=1\,\text{mod}\, 2}}^{s'+s-1}
 e^{-\pi i (\rho+i\epsilon)((2\ell'+3-r-r')\alpha_+-2\nu)} \Bigl(e^{-\pi i \alpha_-\ell (\rho+i\epsilon)} -e^{\pi i \alpha_-\ell (\rho+i\epsilon)} \Bigr)d\rho\\
&= \sum_{\ell'=0}^\infty 
\sum_{\substack{\ell=|s'-s|+1\\ s+s'+\ell=1\,\text{mod}\, 2}}^{s'+s-1}\Bigl(\delta(\alpha_{r+r'-2\ell'-2,-\ell+1}-\nu)-\delta(\alpha_{r+r'-2\ell'-2,\ell+1}-\nu)\Bigr)
 \end{split}
\end{equation*}

Now using $\alpha_{r,s}+\alpha_0/2=\alpha_{r-1,s-1}$ and $\alpha_{r,s}=\alpha_{r+1,p+s}$
and the expression of atypical characters in terms of typicals,
we get 
\begin{equation*}
 \begin{split}
\mathrm{ch}[M_{r,s}^\epsilon]\times  \mathrm{ch}[M_{r',s'}^\epsilon] &= 
\sum_{\ell'=0}^\infty 
\sum_{\substack{\ell=|s'-s|+1\\ s+s'+\ell=1\,\text{mod}\, 2}}^{s'+s-1}\Bigl(\mathrm{ch}[F_{\alpha_{r+r'-1-2\ell'-1,p-\ell}}^\epsilon]-
\mathrm{ch}[F_{\alpha_{r+r'-1-2\ell'-2,\ell}}^\epsilon\Bigr)\\
&=\sum_{\substack{\ell=|s-s'|+1\\ \ell+s+s'=1\, \mathrm{mod}\, 2}}^{s+s'-1}\mathrm{ch}[M_{r+r'-1,\ell}^\epsilon]
\end{split}
\end{equation*}
In the case of $p<s+s'-1 \leq 2p-1$ we have to use Proposition \ref{prop:rel} 
to obtain
\begin{equation*}
 \begin{split}
\mathrm{ch}[M_{r,s}^\epsilon]\times  \mathrm{ch}[M_{r',s'}^\epsilon] &= 
 \sum_{\substack{\ell=p+1\\ \ell+s+s'=1\, \mathrm{mod}\, 2}}^{s+s'-1}\Bigl(\mathrm{ch}[M_{r+r'-2,\ell-p}^\epsilon]+
\mathrm{ch}[M_{r+r'-1,2p-\ell}^\epsilon]+\mathrm{ch}[M_{r+r',\ell-p}^\epsilon]\Bigr)+\\
&\quad\sum_{\substack{\ell=|s-s'|+1\\ \ell+s+s'=1\, \mathrm{mod}\, 2}}^{p}\mathrm{ch}[M_{r+r'-1,\ell}^\epsilon].
\end{split}
\end{equation*}
We also have to check that the following relation inside the Verlinde algebra of characters 
$${\rm ch}[{F}^\epsilon_{\alpha_{r-1,p-s}}]={\rm ch}[M_{r,s}^\epsilon]+{\rm ch}[M_{r-1,p-s}^\epsilon]$$
is consistent with the proposed multiplication. This follows immediately from the relation 
\begin{equation*} 
\begin{split}
{\rm ch}[{F}_{\alpha_{r-1, p-s}}^\epsilon] \times {\rm ch}[{F}_{\nu}^\epsilon] &= 
({\rm ch}[M_{r,s}^\epsilon]+{\rm ch}[M_{r-1,p-s}^\epsilon]) \times {\rm ch}[{F}_{\mu}^\epsilon] \\
&=\sum_{l=-s+2; 2}^s  {\rm ch}[{F}_{\mu+\alpha_{r,l}}^\epsilon]+\sum_{l=-(p-s)+2; 2}^{p-s} {\rm ch} [{F}^\epsilon_{\mu+\alpha_{r-1,l}}]\\
&= \sum_{l=0}^{p-1} {\rm ch}[{F}_{\alpha_{r-1,p-s-2l}+\mu}^\epsilon].
\end{split}
\end{equation*}
Commutativity is clear from the definition, while associativity can be easily checked directly (It also follows from Theorem \ref{Q-thm} below). 
This completes the proof of the theorem. \eop
\end{proof}


The previous computations with the characters is very useful because it gives us lots of hints about the fusion rules 
between triples of modules for the singlet. Although we do not have a proof that the category of $\mathcal{W}(2,2p-1)$-Mod
is a braided tensor category, we believe this to be the case (or at least a suitable sub-category). Thus we can talk about its Grothendieck ring. 
\begin{conjecture} \label{fusion-rules} 
The relations in Theorem \ref{thm:verlinde} also hold inside the Grothendieck ring.
\end{conjecture}
A vertex operator algebra approach to this conjecture will be subject of \cite{AdM4}.

\subsection{Regularized quantum dimensions}

We introduce the regularized quantum dimension of a module for ${\rm Re}(\epsilon)<0$ as 
\begin{equation}
 \qdim{V^\epsilon}= \lim_{\tau\rightarrow 0+}\frac{\text{ch}[V^\epsilon(\tau)]}{\text{ch}[M_{1,1}^\epsilon](\tau)}.
\end{equation}
They are
\begin{proposition} \label{q-dim-reg}
The regularized quantum dimension of typical characters are
\[
\qdim{F_\lambda^\epsilon}= q_{\epsilon}^{2\lambda-\alpha_0}\frac{\mathrm{sin}(-\pi\alpha_+\epsilon i)}{\mathrm{sin}(\pi\alpha_-\epsilon i)}=
q_{\epsilon}^{2\lambda-\alpha_0}\sum_{\substack{\ell=-p+1\\ \ell+p=1\,\mathrm{mod}\, 2}}^{p-1}  q_{\epsilon}^{\alpha_-\ell}
\]
and of atypicals they are
\[
\qdim{M_{r,s}^\epsilon}= q_{\epsilon}^{-(r-1)\alpha_+}\frac{\mathrm{sin}(\pi s\alpha_-\epsilon i)}{\mathrm{sin}(\pi\alpha_-\epsilon i)}=
q_{\epsilon}^{-(r-1)\alpha_+}\sum_{\substack{\ell=-s+1\\ \ell+s=1\,\mathrm{mod}\, 2}}^{s-1}  q_{\epsilon}^{\alpha_-\ell}
\]
for $q_\epsilon=e^{\pi\epsilon}$.
\end{proposition}
\begin{proof}
We have the limits
\[
\lim_{\tau\rightarrow 0}\eta(q) \text{ch}[F_\lambda^\epsilon]=e^{2\pi \epsilon(\lambda-\alpha_0/2)}=q_{\epsilon}^{2\lambda-\alpha_0}
\]
and  for positive $a$
\[
\lim_{\tau\rightarrow 0}P_{a\epsilon}(b\tau;c\tau)= \sum_{n=0}^\infty e^{2\pi\epsilon a(n+1/2)}=(e^{-a\pi\epsilon}-e^{a\pi\epsilon})^{-1}
\]
Observe that here it was essential that ${\rm Re}(\epsilon)<0$.  
\[
\lim_{\tau\rightarrow 0}\eta(q) \text{ch}[M_{r,s}^\epsilon]=\frac{e^{-\beta_{r,s}^-2\pi\epsilon}-e^{-\beta_{r,s}^+2\pi\epsilon}}{e^{-a\pi\epsilon}-e^{a\pi\epsilon}}
= e^{-(r-1)\alpha_+\epsilon}\frac{\mathrm{sin}(\pi s\alpha_-\epsilon i)}{\mathrm{sin}(-\pi\alpha_+\epsilon i)}.
\]
The quantum dimensions follow. The sum expansion of the quotients of sin$(x)$ are as in the proof of Theorem \ref{thm:verlinde}. \eop
\end{proof}

The regularized quantum dimensions should be regarded as functions of the regularization parameter $\epsilon$ with ${\rm Re}(\epsilon)<0$. Consider the vector space $\mathcal{Q}$ spanned by (regularized) quantum dimensions of atypical and typical 
modules. Pointwise multiplication of quantum dimensions defines a commutative product on  $\mathcal{Q}$, which compares nicely to the Verlinde algebra.
\begin{theorem} \label{Q-thm}
The algebra of regularized quantum dimensions $\mathcal{Q}$ is isomorphic to the Verlinde algebra $\mathcal{V}_{ch}$.  
\end{theorem}
\begin{proof}
The products of regularized quantum dimensions are 
\begin{equation}\nonumber
\begin{split}
\qdim{ F_\lambda^\epsilon} \times \qdim{ F_\mu^\epsilon}&=
q_{\epsilon}^{2\lambda+2\mu-\alpha_0}\frac{\mathrm{sin}(-\pi\alpha_+\epsilon i)}{\mathrm{sin}(\pi\alpha_-\epsilon i)}
\sum_{\substack{\ell=-p+1\\ \ell+p=1\,\mathrm{mod}\, 2}}^{p-1}  q_{\epsilon}^{\alpha_-\ell-\alpha_0}\\
&= \sum_{\ell=0}^{p-1} \frac{\mathrm{sin}(-\pi\alpha_+\epsilon i)}{\mathrm{sin}(\pi\alpha_-\epsilon i)} q_{\epsilon}^{2\lambda+2\mu+2\alpha_-\ell-\alpha_0}
= \sum_{\ell=0}^{p-1}\qdim{F_{\lambda+\mu+\ell\alpha_-}^\epsilon}\\
\qdim{M_{r,s}^\epsilon} \times \qdim{F_\mu^\epsilon}&=
q_{\epsilon}^{2\mu-\alpha_0-(r-1)\alpha_+}\frac{\mathrm{sin}(-\pi\alpha_+\epsilon i)}{\mathrm{sin}(\pi\alpha_-\epsilon i)}
\sum_{\substack{\ell=-s+1\\ \ell+s=1\,\mathrm{mod}\, 2}}^{s-1}  q_{\epsilon}^{\alpha_-\ell-\alpha_0}\\
&\hspace*{-5mm}= \sum_{\substack{\ell=-s+2\\ \ell+s=0\,\mathrm{mod}\, 2}}^{s} \frac{\mathrm{sin}(-\pi\alpha_+\epsilon i)}{\mathrm{sin}(\pi\alpha_-\epsilon i)} 
q_{\epsilon}^{2\mu-(r\alpha_++\ell\alpha_--\alpha_0)-\alpha_0}
= \sum_{\substack{\ell=-s+2\\ \ell+s=0\,\mathrm{mod}\, 2}}^{s}\qdim{F_{\mu+\alpha_{r,\ell}}^\epsilon}
\end{split}
\end{equation}
\begin{equation}\nonumber
\begin{split}
\qdim{M_{r,s}^\epsilon} \times \qdim{M_{r',s'}^\epsilon}&=
 \frac{q_{\epsilon}^{-(r+r'-2)\alpha_+}}{\mathrm{sin}(\pi\alpha_-\epsilon i)}\frac{q_{\epsilon}^{-s\alpha_-}-q_{\epsilon}^{s\alpha_-}}{2i}
\sum_{\substack{\ell=-s'+1\\ \ell+s'=1\,\mathrm{mod}\, 2}}^{s'-1}  q_{\epsilon}^{\alpha_-\ell}\\
&\hspace*{-10mm}= \sum_{\substack{\ell=-s'+1\\ \ell+s'=1\,\mathrm{mod}\, 2}}^{s'-1}  q_{\epsilon}^{-(r+r'-2)\alpha_+}\frac{\mathrm{sin}(\pi\alpha_-(\ell+s)\epsilon i)}{\mathrm{sin}(\pi\alpha_-\epsilon i)}
=\sum_{\substack{\ell=|s-s'|+1\\ \ell+s+s'=1\, \mathrm{mod}\, 2}}^{s+s'-1}\qdim{M_{r+r'-1,\ell}^\epsilon}.
\end{split}
\end{equation}
Here $\qdim{M_{r,s}^\epsilon}$ for $s>p$ is defined as
\[  \lim_{\tau\rightarrow 0+}\frac{\text{ch}[M_{r,s}^\epsilon(\tau)]}{\text{ch}[M_{1,1}^\epsilon](\tau)}\]
so that by Proposition \ref{prop:rel} 
\[ \qdim{M_{r,s}^\epsilon} = \qdim{M_{r-1,s-p}^\epsilon}+\qdim{M_{r,2p-s}^\epsilon}+\qdim{M_{r+1,s-p}^\epsilon}\]
since all limits involved exist. 
It follows that
\begin{equation}\nonumber
\begin{split}
\qdim{M_{r,s}^\epsilon} \times \qdim{M_{r',s'}^\epsilon}&=
 \sum_{\substack{\ell=|s-s'|+1\\ \ell+s+s'=1\, \mathrm{mod}\, 2}}^{s+s'-1}\qdim{M_{r+r'-1,\ell}^\epsilon}
\end{split}
\end{equation}
for $s+s'-1\leq p$ and
\begin{equation}\nonumber
\begin{split}
\qdim{M_{r,s}^\epsilon} \times &\qdim{M_{r',s'}^\epsilon}=
 \sum_{\substack{\ell=|s-s'|+1\\ \ell+s+s'=1\, \mathrm{mod}\, 2}}^{s+s'-1}\qdim{M_{r+r'-1,\ell}^\epsilon}+\\
& \sum_{\substack{\ell=p+1\\ \ell+s+s'=1\, \mathrm{mod}\, 2}}^{s+s'-1}\Bigl(\qdim{M_{r+r'-2,\ell-p}^\epsilon}+
\qdim{M_{r+r'-1,2p-\ell}^\epsilon}+\qdim{M_{r+r',\ell-p}^\epsilon}\Bigr).
\end{split}
\end{equation}
for $s+s'-1>p$.

Observe also the relation 
$$\qdim {{F}^\epsilon_{\alpha_{r-1,p-s}}}=\qdim{M_{r,s}^\epsilon}+\qdim{M_{r-1,p-s}^\epsilon},$$
which is true as all three limits involved exist. 
The rest of the proof is the observation that the linear map from $\mathcal{V}_{ch}$ to $\mathcal{Q}$ induced by
$${\rm ch}(X^\epsilon) \mapsto \qdim {X^\epsilon}$$
is one-to-one, which follows easily by using the linear independence of power functions $q_\epsilon^{\nu}$.
\eop
\end{proof}

\begin{remark}
 {\em In rational conformal field theory, the map from the Verlinde algebra to quantum dimensions is an algebra homomorphism
 with non-trivial kernel. The information about fusion rules obtained from quantum dimension is then very limited. 
 In our case the quantum dimension
 in the limit $\epsilon\rightarrow 0$ of typical modules $F_\lambda$ is $p$ and of atypicals $M_{r,s}$ is $s$.
 On the other hand the regularized quantum dimension as a function of $\epsilon$ contains the same information as the Verlinde algebra.
 }
\end{remark}

\begin{remark}
{\em Recall the fusion rules for atypical representations of the Virasoro vertex algebra $L(c_{p,1},0)$ (cf. \cite{L}, \cite{Fl}, and \cite{M}):
$$L(c_{p,1},h_{r,s}) \times L(c_{p,1},h_{r',s'})=\sum_{r'' \in A(r,r'), s'' \in A(s,s')} L(c_{p,1},h_{r'',s''}),$$
where we assume that all indices are positive and $A_{i,j}=\{ i+j-1, i+j-3,\cdots, |i-j|+1 \}.$
This formula merely indicates the fusion rules among triples of irreducibles, defined as dimensions of the space of 
intertwining operators \cite{HLZ}, and should not be viewed as a 
relation in the (hypothetical) Grothendieck ring.
If we view $L(c_{p,1},h_{r'',s''})$ as the top (summand) component of an atypical singlet module, the above fusion rules can be used 
to give an upper bound for the singlet fusion rules. But this "upper bound" is of course different compared to the proposed fusion rules in Conjecture \ref{fusion-rules}. }
\end{remark}

\begin{remark}
{\em As shown in \cite{BM},  partial and false theta functions admit higher rank generalizations coming from higher rank ADE-type Lie algebras in a way that $\mathfrak{g}=sl_2$ recovers the functions studied in this paper. In the same paper, various properties of characters of modules and their quantum dimensions are studied. In particular, we expect their $q$-dimensions to be positive integers. In \cite{CMW} we also study the $(p_+,p_-)$ singlet algebra and the supertriplet introduced in \cite{AdM2}. 
}
\end{remark}

\hspace*{1cm}

\noindent Department of Mathematical and Statistical Sciences, University of Alberta,
Edmonton, Alberta  T6G 2G1, Canada. 
\emph{email: creutzig@ualberta.ca}

\hspace*{1cm}

\noindent Department of Mathematics and Statistics, SUNY-Albany, 1400 Washington Avenue, Albany, NY 12222, USA.
\emph{email: amilas@albany.edu}

\end{document}